\documentclass[conference]{ieeeconf}

\IEEEoverridecommandlockouts
\usepackage[svgnames,dvipsnames]{xcolor}
\usepackage{textcomp}
\usepackage{amssymb, amsfonts, mathtools, xargs, tensor, units, cite, stmaryrd, mathrsfs}
\usepackage{algorithm}
\usepackage{algpseudocode}
\usepackage{graphicx}
\usepackage{booktabs}

\usepackage[unicode=true, hidelinks]{hyperref}

\newcommand{\proj}{\operatorname{proj}}
\newcommand{\SGN}{\operatorname{SGN}}

\newtheorem{assumption}{Assumption}

\newtheorem{remark}{Remark}

\newtheorem{lemma}{Lemma}

\newtheorem{theorem}{Theorem}

\usepackage{tikz}
\usepackage{pgfplots}
\usepackage{pgfplotstable}
\usepgfplotslibrary{colorbrewer}
\pgfplotsset{compat=1.18}
\usepgfplotslibrary{colorbrewer}
\usetikzlibrary{spy}
\usetikzlibrary{positioning,calc}
\definecolor{cmgood}{RGB}{76,175,80}
\definecolor{cmbad}{RGB}{239,83,80}
\definecolor{cmneutral}{RGB}{245,245,245}
\definecolor{cmtext}{RGB}{40,40,40}

\begin{document}

\title{\LARGE \bf Adaptive Control with Sparse Identification of Nonlinear Dynamics}
\author{Trivikram Satharasi$^{1}$, Tochukwu E. Ogri$^{1}$, Muzaffar Qureshi$^{1}$, Kyle Volle$^{2}$, Rushikesh Kamalapurkar$^{1}$
\thanks{This research was supported in part by the Air Force Research Laboratory under contract number FA8651-24-1-0019. Any opinions, findings, or recommendations in this article are those of the author(s), and do not necessarily reflect the views of the sponsoring agencies.}%
\thanks{$^{1}$ Department of Mechanical and Aerospace Engineering, University of Florida, email: {\tt\footnotesize \{t.satharasi, tochukwu.ogri, muzaffar.qureshi, rkamalapurkar\}@ufl.edu}.}%
\thanks{$^{2}$ Torch Technologies, Shalimar, FL, USA, email: {\tt \footnotesize Kyle.Volle@torchtechnologies.com.}}}
\maketitle
\thispagestyle{empty}
\pagestyle{empty}

\begin{abstract}
This paper develops a sparsity-promoting integral concurrent learning (SP-ICL) adaptation law for a linearly parametrized uncertain nonlinear control-affine systems. The unknown parameters are learned using ICL with sparsity-promoting $\ell_1$ regularization. The use of $\ell_1$ regularization for sparsity promotion is common in system identification and machine learning; however, unlike existing approaches, this paper develops an online parameter update law that integrates the regularization penalty with ICL via sliding modes. We show via non-smooth Lyapunov analysis that the trajectories of the closed-loop system are ultimately bounded under the SP-ICL update law. Simulations verify the effectiveness of the sparsity penalty in the SP-ICL update law on recovering sparse dynamics during trajectory tracking. 
\end{abstract}

\section{Introduction}

The performance of most control systems depends on having accurate mathematical models to ensure stability and achieve desired closed-loop behavior. These models are typically derived from fundamental physical laws, such as Euler-Lagrange equations based on Newtonian mechanics, or satellite orbital and attitude dynamics based on gravitational theory \cite{Brunton2016}. However, in practice, such models often include parameters that are not exactly known. For instance, in autonomous driving applications, the friction coefficient varies with tire properties and road conditions, while air density fluctuates with changes in temperature and humidity.

These sources of uncertainty have led to the development of adaptive control approaches, including recursive least squares and model reference adaptive control \cite{SCC.Ioannou.Sun1996}, which estimate the unknown parameters using Lyapunov-derived adaptive update laws to guarantee parametric convergence. However, these adaptive approaches usually require the persistent excitation (PE) condition \cite{SCC.Astrom.Wittenmark1989, SCC.Krstic.Kanellakopoulos.ea1995, SCC.Ioannou.Sun1996, SCC.Kamalapurkar2017, SCC.Pan.Shi.ea2024}, which is often difficult to satisfy in practice.

To avoid the PE assumption, recent adaptive control methods ensure parameter convergence under weaker, verifiable conditions. One such method is dynamic regressor extension and mixing (DREM), which transforms the original vector regression problem into multiple decoupled scalar regressions, enabling convergence with reduced excitation \cite{SCC.Pyrkin.Bobtsov.ea2019}. Another approach, memory regressor extension (MRE), incorporates state history alongside current measurements, allowing convergence under relaxed conditions such as interval or finite excitation \cite{SCC.Katiyar.BasuRoy.ea2022, SCC.Pan.Shi.ea2024, SCC.Qureshi.Ogri.ea2025, SCC.Parikh.Kamalapurkar.ea2019, SCC.Ogri.Bell.ea2023, SCC.Xing.Na.ea2025}.

Despite these advancements in adaptive control, the methods discussed above require a correct mathematical model structure to guarantee convergence and cannot handle cases where the underlying dynamics are unknown or uncertain. For uncertain dynamical systems, controllers can be designed by first identifying the system dynamics from collected data, such as input-output measurements or state trajectories, and then using the identified model for control design. Several data-driven methods have been developed for system identification, including the autoregressive models\cite{Akaike1998}, dynamic mode decomposition (DMD)\cite{Kutz2016}, eigensystem realization algorithm (ERA)\cite{Juang1985}, Gaussian Processes \cite{SCC.Csato.Opper2002,SCC.Qureshi.Ogri.ea2024}, Koopman operator-based approaches\cite{Mezic2005}, Nonlinear Autoregressive Moving Average with Exogenous Inputs (NARMAX) models \cite{billings.2013.Narmax}, Volterra series\cite{Brockett1975}, and neural networks. While these approaches enable control design without prior knowledge of the underlying dynamics, most of these models are trained offline using historical data. 

Online adaptive control methods using neural networks address this limitation by integrating parameter adaptation into a Lyapunov-based framework, guaranteeing tracking performance \cite{Patil.Le.2021, Hart.Patil.2025, SCC.Shen.Xue.ea2025}. While these methods allow real-time adaptation, they rely on the selected network structure and typically produce models that approximate the observed dynamics rather than recovering the true governing equations, with validity generally limited to the regions covered by the training data and a risk of overfitting under environmental changes such as varying friction. These limitations motivate data-driven sparse identification techniques such as Sparse Identification of Nonlinear Dynamics (SINDy) \cite{Brunton2016}, which aim to discover sparse models that may reveal the underlying physical laws directly from observed trajectories.

The SINDy algorithm \cite{Brunton2016} leverages the observation that many nonlinear dynamical systems can be modeled on only a few active basis functions. SINDy identifies these active basis functions from a large library of candidate basis functions using sparsity-promoting techniques, such as $\ell_1$-regularization \cite{SCC.Istiqphara.Wahyunggoro.ea2024, SCC.MedaCampana.Sanchez.ea2025}, or weaker formulations that incrementally reconstruct features for regression \cite{SCC.Russo.Laiu.ea2025}. Computationally efficient variants, such as Ensemble-SINDy, have also been developed \cite{Fasel2022}. While SINDy can be used online within an Model Predictive Control framework \cite{Fasel2021, Fasel2022}, existing sparse regression methods have not been studied from an adaptive control perspective with stability and convergence guarantees.

This paper presents a sparsity-promoting integral concurrent learning (SP-ICL) method for the adaptive identification of nonlinear system dynamics. By incorporating $\ell_1$-regularization into an integral concurrent learning (ICL) framework, the approach promotes sparsity during online parameter estimation. In contrast to standard sparse identification algorithms (e.g., SINDy) that predominantly operate offline without closed-loop stability guarantees, SP-ICL embeds the identification process within a Lyapunov-based adaptive control design. Because the $\ell_1$ penalty is non-smooth, the resulting parameter update law is governed by a differential inclusion. Employing Lyapunov methods for non-smooth systems, we derive sufficient conditions to establish convergence of tracking and parameter estimation errors. The resulting framework yields sparse models of nonlinear dynamics suitable for real-time implementation. 

Due to the set-valued nature of the generalized gradient of the $\ell_1$ cost function, incorporating a projection operator \cite{SCC.Ioannou.Sun1996, SCC.Ogri.Bell.ea2023} to enforce boundedness of the parameter estimates, is challenging. For example, the projected differential inclusion must satisfy regularity conditions such as upper semicontinuity, which can be nontrivial to verify \cite{SCC.Bressan.Cellina1989}. Instead, we exploit properties of the set-valued signum function to show that by projecting just the continuous part of the update law, the parameter estimates remain forward invariant with respect to the prescribed compact set. Upper semicontinuity of the resulting update law follows from the continuity of the projection operator.

\section{Problem Formulation}\label{sec:problemformulation}
Consider the nonlinear control-affine system 
\begin{equation}\label{eq:dynamics}
    \dot{x} = Y(x)\theta + g(x)u, \quad x(t_{0}) = x_{0},
\end{equation}
where $x \in \mathbb{R}^n$ is the state, $\theta \in \mathbb{R}^{p}$ is an unknown parameter vector, and $u \in \mathbb{R}^m$ is the control input. The function $Y: \mathbb{R}^{n} \to \mathbb{R}^{n \times p}$ is a known over-complete library of candidate nonlinear basis functions, and $g: \mathbb{R}^{n} \to \mathbb{R}^{n \times m}$ is the known control effectiveness matrix. 

The representation in \eqref{eq:dynamics} assumes that the dynamics admit a decomposition over $Y(x)$ with a sparse coefficient vector $\theta$. The functions $Y$ and $g$ are assumed to be locally Lipschitz continuous.

The objective is to generate a sparse estimate of $\theta$ online while ensuring tracking of a bounded, continuously differentiable desired trajectory 
$x_d: \mathbb{R}_{\geq 0} \to \mathbb{R}^n$ that satisfies 
\begin{equation}
    \|x_d(t)\| \le \bar{x}_{d}, \quad \forall t \ge 0,
\end{equation}
where $\bar{x}_{d} > 0$ is a known constant. The following assumption is made to facilitate the formulation of the tracking controller.
\begin{assumption}\label{ass:pseudo}
The control effectiveness matrix $g$ has full row rank for all $x\in \mathbb{R}^n$, so a locally Lipschitz continuous right pseudoinverse $g^+(x)$ exists with $g(x) g^+(x) = I_n$.
\end{assumption}

The assumption above can be satisfied if the system is fully actuated or overactuated ($m \geq n$). It is made to simplify the control design so that we can focus the discussion on sparse parameter estimation.

To quantify the control objective, define the tracking error $e = x - x_d$ and let $\hat{\theta} \in \mathbb{R}^p$ be an estimate of $\theta$, with $\tilde{\theta} \coloneqq \theta - \hat{\theta}$ denoting the parameter estimation error. Under Assumption~\ref{ass:pseudo}, the controller can be designed via certainty equivalence as
\begin{equation}\label{eq:control}
u = g^+(x)\big(\dot{x}_d - Y(x)\hat{\theta} - K e\big),
\end{equation}
where $K = K^\top \succ 0$ is a constant gain. Using \eqref{eq:control} then yields the tracking error dynamics
\begin{equation}\label{eq:closed_error}
\dot{e} = -K e + Y(x)\tilde{\theta}.
\end{equation}

\section{Parameter Estimator Design}
\label{section:parameterEstimator}

In this section, an SP-ICL update law is developed to estimate the unknown parameters.

\begin{lemma}\label{lem:ErrorTermformulation}
Given a fixed delay $T>0$, the filtered regressor $Y_f$ and filtered state $U_f$  defined as
\begin{equation*}
    Y_f(t) \coloneqq 
    \begin{cases} 
      0_{n \times p}, & t < t_0+T, \\
      \int_{t-T}^{t} Y(x(\tau))\,d\tau, & t \ge t_0 +T,
    \end{cases}
\end{equation*}
and
\begingroup\medmuskip=0mu\begin{equation*}
    U_f(t) \coloneqq 
    \begin{cases} 
      0_{n \times 1}, & t < t_0 + T,\\
      x(t) - x(t - T) - \int_{t-T}^{t} g(x(\tau))u(\tau)\,d\tau, & t \ge t_0 + T,
    \end{cases}
\end{equation*}\endgroup
satisfy
\begin{equation}\label{eq:incrementalRelation}
    U_f(t) \;=\; Y_f(t)\theta,
\end{equation}
for all $t\ge t_0$.
\end{lemma}
\begin{proof}
The result follows directly from the Fundamental Theorem of Calculus applied to \eqref{eq:dynamics}.
\end{proof}

To enable online parameter adaptation and data reuse, a finite set of delayed samples of size $N \ge 1$ is maintained in a \emph{history stack}, defined as $\mathcal{H} \coloneqq \big\{ {U_f}_i,\, {Y_f}_i \big\}_{i=1}^{N}$, where ${U_f}_i = U_f(t_i)$ and ${Y_f}_i = Y_f(t_i)$ for discrete times $t_i > T$, are updated according to Algorithm~\ref{algo:dataselection}.
We define the memory regression extension signals $\mathcal{Y}$ and $\mathcal{U}$ such that,
\begin{equation}\label{eq:YU}
    \mathcal{Y} = \sum_{i=1}^N \frac{{Y_f}_i^\top {Y_f}_i}{1+\kappa \lVert {Y_f}_i \rVert^2} \,
    \text{ and } \,
    \mathcal{U} = \sum_{i=1}^N \frac{{Y_f}_i^\top {U_f}_i}{1+\kappa \lVert {Y_f}_i \rVert^2},
\end{equation}
where $\left({Y_f}_i, {U_f}_i\right) \in \mathcal H$, and $\kappa>0$ is a user-selected constant gain. The closed-loop system is assumed to satisfy the following excitation condition.

\begin{assumption}\label{ass:fullRank}
The filtered regressor $Y_f:[t_0, \, \infty)\to\mathbb{R}^{n\times p}$ is \emph{finitely exciting} (FE) over some interval $[t_0, t_{s}]$, i.e., there exists a finite collection of sampling instants $\{t_i\}_{i=1}^{N} \subset [t_{0}, t_{s}]$ such that for all $t \geq t_{s}$
\begin{equation}\label{eq:rankCond}
    \lambda_{\min}(\mathcal{Y})
    \ge \underline{y} > 0.
\end{equation}
\end{assumption}
\begin{remark}
The assumption is commonly adopted in ICL-based parameter estimation frameworks (\cite{SCC.Chowdhary.Johnson2010, SCC.Parikh.Kamalapurkar.ea2019,  SCC.Ogri.Bell.ea2023, SCC.Satharasi.Ogri.ea2026}). Unlike the classical PE condition, however, it requires excitation only over a finite interval, making it a weaker requirement.
\end{remark}

A history stack satisfying \eqref{eq:rankCond} is called \emph{full rank}. The constant $\underline{y}$ quantifies the information richness of the stored data and determines the convergence rate of the parameters, as seen in the subsequent analysis. To guarantee sufficient excitation and improve estimator conditioning and alertness, the history stack $\mathcal{H}$ is updated using the strategy in Algorithm~\ref{algo:dataselection}. 

\begin{algorithm}
\caption{ICL Data Selection}
\label{algo:dataselection}
\begin{algorithmic}[1]
\Require Number of data points $N$, target $\underline{y} > 0$, improvement threshold $\Delta \geq 1$
\State Store the first $N$ points: $\{U_{f_i}, Y_{f_i}\}_{i=1}^{N}$
\For{each new datum $(U_{f}(t), Y_f(t))$}
    \State $\mathcal{Y} \gets \sum_{i} \frac{Y_{f_i}^{\top}Y_{f_i}}{1+ \nu \|Y_{f_i}\|^2}$ 
    \If{$\lambda_{\min}\{\mathcal{Y}\} \geq \underline{y}$}
        \Comment{Favor new data if eigenvalue target is already met}
        \For{$j=1$ to $N$}
            \State $y_j \gets \sum_{i\neq j} \frac{Y_{f_i}^{\top}Y_{f_i}}{1+ \nu \|Y_{f_i}\|^2} + \frac{Y_f^{\top}(t)Y_f(t)}{{1+ \nu \|Y_{f}\|^2}}$
            \If{$\lambda_{\min}\{y_j\} \geq \underline{y}$}
                \Comment{Search from oldest to newest}
                \For{$k=j$ to $N-1$}
                    \State $Y_{f_k} \gets Y_{f(k+1)},\; U_{f_k} \gets U_{f(k+1)}$
                \EndFor
                \Comment{Shift data to make room for new datum}
                \State $Y_{f_N} \gets Y_f(t),\; U_{f_N} \gets U_{f}(t)$
                \State \textbf{break}
            \EndIf
        \EndFor
    \Else
        \Comment{If the target is not met, then maximize the minimum eigenvalue}
        \For{$j=1$ to $N$}
            \State $y_j \gets \sum_{i\neq j} \frac{Y_{f_i}^{\top}Y_{f_i}}{1+ \nu \|Y_{f_i}\|^2} + \frac{Y_f^{\top}(t)Y_f(t)}{{1+ \nu \|Y_{f}\|^2}}$
            \If{$\lambda_{\min}\{y_j\} > \Delta \cdot \underline{y}$}
                \For{$k=j$ to $N-1$}
                    \State $Y_{f_k} \gets Y_{f(k+1)},\; U_{f_k} \gets U_{f(k+1)}$
                \EndFor
                \Comment{Shift data to make room for new datum}
                \State $Y_{f_N} \gets Y_f(t),\; U_{f_N} \gets U_{f}(t)$
                \State \textbf{break}
            \EndIf
        \EndFor
    \EndIf
\EndFor
\end{algorithmic}
\end{algorithm}

Under Algorithm~\ref{algo:dataselection}, a candidate data point $(U_{\rm new}, Y_{\rm new})$ replaces the oldest possible entry $({U_f}_{j}, {Y_f}_{j})$ in the history stack $H$ if
\begin{equation}\label{eq:eigMaxRuleRewritten}
\lambda_{\min}(\mathcal{Y}^{\ast}) > \underline{y},
\end{equation}
where $\mathcal{Y}^{\ast} = \sum_{i\neq j}^{N} \frac{{Y_f}_i^\top {Y_f}_i}{1 + \kappa \|{Y_f}_i\|^2} + \frac{Y_{\rm new}^\top Y_{\rm new}}{1 + \kappa \|Y_{\rm new}\|^2}$. By construction, \eqref{eq:eigMaxRuleRewritten} guarantees $\lambda_{\min}(\mathcal{Y}) > \underline{y}$ for all $t > t_{s}$, ensuring the history stack remains \emph{full rank} and new data replaces old data whenever possible without violating the target minimum eigenvalue.
\subsection{Sparse Recursive Cost and Generalized Gradient}
\label{subsec:sparse_cost}

To promote sparse parameter estimates while maintaining consistency with the accumulated integral data, consider the $\ell_{1}$-regularized cost
\begin{equation}\label{eq:J}
J(\hat{\theta}) \coloneqq 
\frac{1}{2} \hat{\theta}^\top \mathcal{Y} \hat{\theta} 
- \hat{\theta}^\top \mathcal{U} 
+ \lambda \|\hat{\theta}\|_1,
\end{equation}
where $\lambda>0$ is the sparsity parameter and $\|\hat{\theta}\|_1$ denotes the one-norm of $\hat{\theta}$. The adaptive update law for $\hat{\theta}$ is derived to minimize the regularized cost using gradient descent. Since the $\ell_1-$regularization term is nonsmooth at the origin, a generalized gradient of $J$ is utilized, given by the Clarke sub-differential \cite{SCC.Clarke1975}
\begin{equation}\label{eq:partial_J}
    \partial_{\hat\theta}J(\hat\theta) =  \mathcal{Y}\hat\theta-\mathcal{U}+\lambda\,\SGN(\hat\theta),
\end{equation}
where $\SGN(\hat{\theta})$ denotes the componentwise set-valued sign map, defined for each component $i$ as $\SGN(\hat{\theta}_i) = \big\{ z \in [-1,1] \mid z = \operatorname{sign}(\hat{\theta}_i) \text{ if } \hat{\theta}_i \neq 0 \big\}$.

Define the compact set $\Theta \coloneqq \{\theta \in \mathbb{R}^p : \|\theta\| \, \leq \, r_{\theta}\}$, where $ r_{\theta}>0$ is a known bound such that $\theta \in \Theta$. The SP-ICL update law is designed as a projected differential inclusion using the negative gradient in \eqref{eq:partial_J} as
\begin{equation}\label{eq:theta_dot}
    \dot{\hat{\theta}} \in \proj_\Theta\Big(\hat{\theta}, \phi(x, e, \hat{\theta}), \Gamma\Big) - \lambda \gamma \, \Gamma \SGN(\hat{\theta}) \eqqcolon G(t, \hat{\theta}),
\end{equation}
where $\phi(x, e, \hat{\theta}) \coloneqq \Gamma\Big(Y(x)^\top e + \gamma \big(\mathcal{U} - \mathcal{Y}\hat{\theta}\big)\Big)$, $\gamma>0$ is the ICL-learning gain, $\Gamma = \Gamma^\top \succ 0$ is a diagonal adaptation gain matrix, and $\proj_{\Theta}(\cdot,\cdot,\cdot)$ denotes the smooth projection operator from \cite[Appendix E.1]{SCC.Krstic.Kanellakopoulos.ea1995}, which ensures that $\hat{\theta}(t) \in \Theta \oplus \bar{B}(0, \epsilon) =: \Theta_\epsilon$ for all $t \ge t_0$, in the absence of the signum term in \eqref{eq:theta_dot}, where $\bar{B}(0, \epsilon)$ is the closed Euclidean ball of radius $\epsilon > 0$ centered at the origin, and $\oplus$ denotes the Minkowski sum. In the subsequent lemma, we show that $\hat{\theta}(t) \in  \Theta_\epsilon$, even with the signum term in \eqref{eq:theta_dot}.
\begin{lemma}\label{lem:theta_bound}
 If $\hat{\theta}(t_{0}) \in \Theta_{\epsilon}$, and there exists a solution to \eqref{eq:theta_dot} on $t\in [t_0, t+\delta)$, then $\Theta_{\epsilon}$ is positively invariant with respect to \eqref{eq:theta_dot} on $t\in [t_0, t+\delta)$.
\end{lemma}
\begin{proof}
Let $\partial \Theta_{\epsilon}$ denote the boundary of $\Theta_{\epsilon}$, and let ${T_\Theta}_{\epsilon}(\hat{\theta})$ denote the tangent cone of $\Theta_{\epsilon}$ at $\hat{\theta}$. Because $\Theta_{\epsilon}$ is a Euclidean ball centered at the origin, the outward-pointing unit normal vector at any boundary point $\hat{\theta} \in \partial \Theta_{\epsilon}$ is given by $n(\hat{\theta}) = \frac{\hat{\theta}}{\|\hat{\theta}\|}$.

From \eqref{eq:theta_dot}, and the standard properties of the projection operator onto a closed convex set (see \cite[Lemma E.1]{SCC.Krstic.Kanellakopoulos.ea1995}), for any $\hat{\theta} \in \partial \Theta_{\epsilon}$, we have
\begin{equation}\label{eq:first_ineq}
n(\hat{\theta})^\top \proj_\Theta\left(\hat{\theta}, \phi(x, e, \hat{\theta}), \Gamma\right) \le 0.
\end{equation}

Next, we evaluate the inner product of the normal vector with the set-valued term $v_{\theta} \coloneqq -\lambda \gamma \Gamma \sigma$ in \eqref{eq:theta_dot}, where $\sigma \in \mathrm{SGN}(\hat{\theta})$. For $\hat{\theta} \in \partial \Theta_{\epsilon}$ and all $\sigma \in \mathrm{SGN}(\hat{\theta})$, we have
\begin{multline}\label{eq:second_ineq}
n(\hat{\theta})^\top v_{\theta}
\le -\lambda \gamma\, n(\hat{\theta})^\top \Gamma \sigma
= -\lambda \gamma\, \left( \frac{\hat{\theta}}{\|\hat{\theta}\|} \right)^\top \Gamma \sigma
\\\le -\lambda \gamma\, \lambda_{\min}(\Gamma) \|\hat{\theta}\|_1 \le 0.
\end{multline}
Then, it follows from \eqref{eq:first_ineq} and \eqref{eq:second_ineq} that for all $\rho \in G(t, \hat{\theta})$,
\begin{equation}
n(\hat{\theta})^\top \rho \le 0,
\end{equation}
which implies that, for all $\rho \in G(t, \hat{\theta})$,  $\rho \in {T_\Theta}_{\epsilon}(\hat{\theta})$. Invoking \cite[Theorem.~7]{aubin.cellina.1984}, $\Theta_\epsilon$ is positively invariant. 
\end{proof}

\section{Stability Analysis}
\label{subsec:lyap}

We now analyze the closed-loop system with error states $z = [\,e^\top,\,\tilde{\theta}^\top\,]^\top \in \mathbb{R}^{n+p}$. As shown in Lemma~\ref{lem:theta_bound}, due to the properties of the projection operator, if $\hat{\theta}(t_0) \in \Theta_{\epsilon}$,  and if \eqref{eq:theta_dot} admits a solution, then $\hat{\theta}(t) \in \Theta_{\epsilon}$ for all $t \ge t_0$. Consequently, the parameter estimation error is uniformly bounded in particular, $\|\tilde{\theta}(t)\| \le 2r_{\theta} + \epsilon$ for all $t \ge t_0$.

Define $d(t, e) \coloneqq Y(e + x_{d}(t))\tilde{\theta}(t)$. If $e \in \overline{B}_{e}(r_{e}) \coloneqq \{e \in \mathbb{R}^{n}: \|e\|\leq r_e\}$ for some $r_e > 0$, then $x = e + x_{d}(t) \in \overline{B}(r_{e} + \overline{x}_{d})$.  Since $Y$ is continuous, there exists a constant $\bar Y > 0$ such that $\|Y(x)\| \le \bar Y$ for all $x \in \overline{B}_{x}(r_{e} + \overline{x}_{d}) $. Consequently, the residual is bounded by
\begin{equation}
    \|d(t,e)\| \le   \bar d, \quad \forall e \in \overline{B}_{e}(r_{e}),
\end{equation}
where $\bar d \coloneqq (2  r_{\theta} + \epsilon)\,\bar Y$. The error dynamics \eqref{eq:closed_error} can then be expressed as
\begin{equation}\label{eq:error_vector_field}
    \dot e = -K e + d(t,e),
\end{equation}
where $d(t,e)$ acts as a bounded disturbance.
\begin{theorem}\label{thm:stability}
Consider the closed-loop system given by \eqref{eq:closed_error} and \eqref{eq:theta_dot}. Let $r_e > 0$ and $r > 0$.  Suppose that Assumption~\ref{ass:pseudo} holds and the initial closed-loop state satisfies
\begin{equation}
z(t_0) \in B_{e}(r_{e}) \times \overline{B}_{\tilde{\theta}}(2 r_{\theta} + \epsilon).
\end{equation}
Assume that the feedback gain $K = K^{\top} \succ 0$ is selected sufficiently large such that the gain conditions
\begin{equation}\label{eq:gain_condition_e}
\|e(t_0)\|^2 < r_e^2 \quad \text{and} \quad \frac{\overline{d}^2}{k^2} < r_e^2,
\end{equation}
hold, where $k \coloneqq \lambda_{\min}(K)$. Furthermore, assume that for $t \ge t_{s}$, the history stack is sufficiently rich such that Assumption~\ref{ass:fullRank} holds and the control parameters $\gamma, \lambda, \Gamma$ are selected such that 
\begin{equation}\label{eq:gain_condition_z}
r > \frac{\underline{m}}{\overline{m}}\max\left\{r_e,2 r_{\theta} + \epsilon\right\} \quad \text{and} \quad \frac{\iota}{\alpha} < \frac{\underline{m}}{\overline{m}} r^2,
\end{equation}
where $\alpha \coloneqq \min(k, \gamma \underline{y})$, $\iota \coloneqq \lambda \gamma \sqrt{p} $, $\underline{m} \coloneqq \frac{1}{2}\min(1, \lambda_{\min}(\Gamma^{-1}))$, and $\overline{m} \coloneqq \frac{1}{2}\max(1, \lambda_{\max}(\Gamma^{-1}))$. Then, the following hold:
\begin{enumerate}
    \item  The closed-loop system admits global solutions, and the error state trajectory $t \mapsto e(t)$ remains inside the open ball $B_e(r_e)$ for all $t \in [t_0, t_{s})$.
    \item  At time $t_{s}$, the closed loop state satisfies $z(t_{s}) \in B_z(r)$.
    \item  For all $t \ge t_{s}$, the closed loop state satisfies $z(t) \in B_z(r)$ and converges to a set $\Omega_{ub,z}$, defined as
    \begin{equation}\label{eq:UUB_set}
    \Omega_{ub,z} \coloneqq \left\{ z \in \mathbb{R}^{n+p} \mid \|z\| \le \sqrt{\frac{\overline{m}}{\underline{m}} \left(\frac{\iota}{\alpha} \right)}\right\}.
    \end{equation}
    Furthermore, $\Omega_{ub,z} \subset \overline{B}_{z}(r)$. 
    \end{enumerate}
\end{theorem}
\begin{proof}
\textit{Part 1 - Analysis on the interval $[t_0, t_{s})$:}

 The closed-loop system is described by a differential inclusion. Since the closed-loop set-valued map is measurable, locally essentially bounded, upper semicontinuous, and has nonempty, compact, and convex values, the differential inclusion meets the conditions of \cite[Chap.~2, Thm.~1]{SCC.Aubin.Cellina1984} and as a result, there exists $\delta > 0$ such that starting from the initial condition $(t_0,z(t_0))$, the closed-loop system admits a maximal absolutely continuous solutions on $[t_0,t_0+\delta)$.

Consider the continuously differentiable candidate Lyapunov  function $V_{e}: \mathbb{R}^{n} \to \mathbb{R}$ defined as
\begin{equation}
    V_e(e) \coloneqq \frac{1}{2}e^\top e.
\end{equation}
 The time derivative of the function $t \mapsto V_e(e(t))$ along the trajectory $e(\cdot)$ exists for almost all $t \in [t_0, t_0+\delta)$. Substituting the dynamics from \eqref{eq:closed_error}, we obtain the bound
\begin{align}
\dot{V}_e(e(t)) &= e(t)^\top \dot{e}(t) \nonumber \\
&= e(t)^\top \big(-Ke(t) + d(t,e(t))\big) \nonumber \\
&\le -k\|e(t)\|^2 + \|e(t)\| \|d(t,e(t))\|,
\end{align}
for almost all $t \in [t_0, t_0+\delta)$. Completing the squares,
\begin{equation}
\dot{V}_e(e(t)) \le -\frac{k}{2}\|e(t)\|^2 + \frac{\|d(t,e(t))\|^2}{2k}, \label{eq:Vdot_bound}
\end{equation}
 for almost all $t \in [t_0, t_0+\delta)$.  Since $V_e(e) = \frac{1}{2}\|e\|^2$, the bound in \eqref{eq:Vdot_bound} can be rewritten as the differential inequality
\begin{equation}\label{eq:diff_ineq}
\dot{V}_e(e(t)) \le -k V_e(e(t)) + \frac{\|d(t,e(t))\|^2}{2k}.
\end{equation}

We now show by contradiction that $e$ does not leave $\overline B_e(r_e)$. Assume that $e(t_0) \in B_e(r_e),\, \hat \theta \in \Theta_\epsilon$ and there exists a time $T \in (0, \delta)$ such that $e(t_0+T) \notin \overline{B}_e(r_e)$. 

Because $e(\cdot)$ is continuous and $e(t_0) \in B_e(r_e)$, the Intermediate Value Theorem guarantees the existence of time instances $\epsilon_1 \in (0, T)$ and $\epsilon_2 \in (\epsilon_1, T]$ such that $e(t) \in \overline{B}_e(r_e)$ for all $t \in [t_0, t_0+\epsilon_1]$, and $e(t) \notin \overline{B}_e(r_e)$ for all $t \in (t_0+\epsilon_1, t_0+\epsilon_2)$. Over the interval $[t_0, t_0+\epsilon_1]$, $e$ remains within $\overline{B}_e(r_e)$, meaning the uniform bound $\|d(t,e(t))\| \le \overline{d}$ is valid. Applying the Comparison Lemma \cite[Lemma~3.4]{SCC.Khalil2002} to \eqref{eq:diff_ineq} over $t \in [t_0, t_0+\epsilon_1]$ yields
\begingroup\medmuskip=0mu\begin{equation}\label{eq:comparison_bound}
V_e(e(t)) \le V_e(e(t_0))e^{-k(t-t_0)} + \frac{\overline{d}^2}{2k^2} \left( 1 - e^{-k(t-t_0)} \right).
\end{equation}\endgroup
By substituting $V_e(e) = \frac{1}{2}\|e\|^2$ into \eqref{eq:comparison_bound}, we establish the norm bound on the tracking error $e$ as
\begin{multline}
\|e(t)\|^2 \le \|e(t_0)\|^2 e^{-k(t-t_0)} + \frac{\overline{d}^2}{k^2} \left( 1 - e^{-k(t-t_0)} \right)  \\
\le \max\left\{ \|e(t_0)\|^2, \, \frac{\overline{d}^2}{k^2} \right\}, \quad \forall t \in [t_0, t_0+\epsilon_1]. \label{eq:max_bound}
\end{multline}

Under the gain condition \eqref{eq:gain_condition_e}, the maximum of the squared norm of the initial tracking error and the squared disturbance-to-gain ratio is strictly less than $r_e^2$. Consequently, \eqref{eq:max_bound} implies the existence of a constant $\varpi > 0$ such that $\|e(t)\| \le r_e - \varpi$ for all $t \in [t_0, t_0+\epsilon_1]$. Since $e(t)\notin \overline{B}_e (r_e)$ at $t=t_0+ \epsilon_1$, we conclude that, $t \mapsto e(t)$ is not continuous at $t_0+\epsilon_1$, which is a contradiction. Therefore, our initial assumption must be false, proving that $e(t) \in \overline{B}_e(r_e)$ for all $t \in [t_0, t_0+\delta)$. 

If $\delta\ge t_0- t_s$, then the results hold on $[t_0, t_s]$. If not, then, because the solution of the closed loop system is restricted to the compact set $\overline{B}_e(r_e) \times \Theta_\epsilon$ over its maximal interval of existence, it is precompact, and as a result, can be extended to $t_s$. Therefore, the solution is defined and $e(t) \in \overline{B}_e(r_e), \, \hat \theta \in \Theta_\epsilon$ for all $t \in [t_0, t_{s}]$.

\textit{Part 2 - State Transition at $t = t_{s}$:}

Due to continuity at $t_{s}$, the tracking error satisfies $e(t_{s}) \in \overline{B}_e(r_e)$, while the projection operator ensures $\tilde{\theta}(t_{s}) \in \overline{B}_{\tilde{\theta}}(2 r_{\theta} + \epsilon)$. Hence, the closed-loop state satisfies
\begin{equation}
z(t_{s}) = \begin{bmatrix} e(t_{s}) \\ \tilde{\theta}(t_{s}) \end{bmatrix} \in \overline{B}_e(r_e) \times \overline{B}_{\tilde{\theta}}(2 r_{\theta} + \epsilon).
\end{equation}

\textit{Part 3 - Analysis on the interval $[t_{s}, \infty)$:}

Consider the closed ball $\overline{B}_z(r)\subset\mathbb{R}^{n+p}$ for $r>0$,
where $r>0$ is selected such that \eqref{eq:gain_condition_z} is satisfied. This selection of $r$ ensures that $z(t_s) \in B_z(r)$. For $t \geq t_{s}$, the system operates with a full-rank memory regressor, i.e., such that Assumption~\ref {ass:fullRank} holds.

Consider the candidate Lyapunov function $V: \mathbb{R}^{n + p} \to \mathbb{R}$ defined as
\begin{equation}\label{eq:V}
V(z) \coloneqq \frac{1}{2}\|e\|^2 + \frac{1}{2} \tilde{\theta}^{\top} \Gamma^{-1} \tilde{\theta}.
\end{equation}
which satisfies the bounds $\underline{m}\|z\|^2 \le V(z) \le \overline{m}\|z\|^2$ for all $z \in \mathbb{R}^{n + p}$, where $\underline{m}, \overline{m} > 0$ are constants defined below \eqref{eq:gain_condition_z}.

The closed-loop system is described by a differential inclusion. Since the closed-loop set-valued map is measurable, locally essentially bounded, upper semicontinuous, and has nonempty, compact, and convex values, the differential inclusion meets the conditions of \cite[Chap.~2, Thm.~1]{SCC.Aubin.Cellina1984} and as a result, there exists $\delta > 0$ such that starting from the initial condition $(t_s,z(t_s))$, the closed-loop system admits a maximal absolutely continuous solutions on $[t_s,t_s+\delta)$.

 Let $t \mapsto z(t)$ be one such solution. Let $\vartheta :[t_s,t_s+\delta)\to\mathbb{R}^p$ and $\sigma:[t_s,t_s+\delta)\to\mathbb{R}^p$ be measurable selections that occur due to discontinuous points in the closed loop error system, such that $\varrho(t)\in \operatorname{Proj}_\Theta\big(\hat{\theta}, \phi, \Gamma\big)$ and $ \sigma(t) \in \operatorname{SGN}(\hat{\theta})$, and  $\dot{\hat{\theta}}(t) = \vartheta(t) - \lambda \gamma \Gamma \sigma(t)$, for almost all $t\in[t_s,t_s+\delta)$. Since $\theta$ is constant, we have $\dot{\tilde{\theta}} = -\dot{\hat{\theta}}$. Whenever the time derivative of $t\mapsto V(z(t))$ exists, it satisfies
\begin{multline}
\dot{V}(z(t)) = -e^\top(t) K e(t) + e^\top(t) Y(x(t))\tilde{\theta}(t) \\- \tilde{\theta}^\top(t) \Gamma^{-1} \big( \vartheta(t) - \lambda \gamma \Gamma \sigma(t) \big). \label{eq:Vdot_z}
\end{multline}

Applying the fundamental property of the projection operator, $\tilde{\theta}^\top \Gamma^{-1} \vartheta \le \tilde{\theta}^\top \Gamma^{-1} \phi$. Substituting the definition of $\phi$ in \eqref{eq:theta_dot}, we have
\begin{multline}
\tilde{\theta}^\top \Gamma^{-1} \phi = \tilde{\theta}^\top \Gamma^{-1} \Gamma \Big(Y(x)^\top e + \gamma(\mathcal{U} + \mathcal{Y}\hat{\theta}) \Big) \\= \tilde{\theta}^\top Y(x)^\top e + \gamma \tilde{\theta}^\top (\mathcal{U} - \mathcal{Y}\hat{\theta}).
\end{multline}
By the design of the ICL history stack, $\mathcal{U} = \mathcal{Y}\theta$. Consequently, $\mathcal{U} - \mathcal{Y}\hat{\theta} = \mathcal{Y}(\theta - \hat{\theta}) = \mathcal{Y}\tilde{\theta}$. Substituting these relations back into \eqref{eq:Vdot_z} yields
\begin{equation}
\dot{V}(z(t)) \le -e^\top(t) K e(t) - \gamma \tilde{\theta}^\top(t) \mathcal{Y} \tilde{\theta}(t) + \lambda \gamma \tilde{\theta}^\top(t) \sigma(t), \label{eq:Vdot_z}
\end{equation}
for almost all $t\in [t_s,t_s + \delta)$. Applying Assumption~\ref{ass:fullRank} and since $\|\sigma(t)\| \leq \sqrt{p}$, we obtain
\begin{equation}\label{eq:Vdot_final}
\dot{V}(z(t)) \le -\alpha\|z(t)\|^2 + \iota \le -\frac{\alpha}{\overline{m}}V(z(t)) + \iota.
\end{equation}
where $\alpha, \iota  > 0$ are constants defined below \eqref{eq:gain_condition_z}.

To establish the ultimate boundedness of the closed-loop state, we again proceed by contradiction.

For the sake of contradiction, assume that there exists some time $t^* > 0$ such that $z(t_{s}+t^*) \notin \overline{B}_z(r)$. Since $z(\cdot)$ is continuous and our selection of $r$ ensures that $z(t_s)\in B_z(r)$, by the Intermediate Value Theorem, there exist time instances $\rho_1 \in (0, t^*)$ and $\rho_2 \in (\rho_1, t^*]$ such that $z(t) \in \overline{B}_z(r)$ for all $t \in [t_{s}, t_{s}+\rho_1]$ and $z(t) \notin \overline{B}_z(r)$ for all $t \in (t_{s}+\rho_1, t_{s}+\rho_2)$.

Applying the Comparison Lemma to \eqref{eq:Vdot_final} over $t \in [t_{s}, t_{s}+\rho_1]$ and using the bounds on $V(z)$, we obtain
\begin{equation}
\|z(t)\|^2 \le \frac{\overline{m}}{\underline{m}}\|z(t_{s})\|^2 e^{-\frac{\alpha}{\overline{m}}(t-t_{s})} + \frac{\overline{m}}{\underline{m}} \frac{\iota}{\alpha} \left( 1 - e^{-\frac{\alpha}{\overline{m}}(t-t_{s})} \right).
\end{equation}
In particular, we have that
\begingroup\medmuskip=0mu\begin{equation}\label{eq:max_bound_z}
\|z(t)\|^2 \le \frac{\overline{m}}{\underline{m}} \max\left\{ \|z(t_{s})\|^2, \, \frac{\iota}{\alpha} \right\}, \, \forall t \in [t_{s}, t_{s}+\rho_1].
\end{equation}\endgroup

Provided the control parameters are selected according to \eqref{eq:gain_condition_z}, the bounds imply the existence of a constant $\varsigma > 0$ such that $\|z(t)\| < r - \varsigma$ for all $t \in [t_{s}, t_{s}+\rho_1]$. As a result, $t \mapsto z(t)$ is not continuous at $t_{s}+\rho_1$, which is a contradiction. 

Therefore, our assumption that there exists some time $t^* > 0$ such that $z(t_{s}+t^*) \notin \overline{B}_z(r)$ is false. Using completeness of pre-compact solutions\cite[Prop.~2]{SCC.Ryan1990} we conclude that $\delta = \infty$ and $z(t) \in \overline{B}_z(r)$ for all $t \ge t_{s}$. 

As $t \to \infty$, the state exponentially converges to the ultimate bound
\begin{equation}
\limsup_{t \to \infty} \|z(t)\| \le \sqrt{ \frac{\overline{m}}{\underline{m}} \left(\frac{\iota}{\alpha}\right) }.
\end{equation}
Since the solution $z(\cdot)$ being analyzed was selected arbitrarily, the ultimate bound holds for all solutions of the closed-loop differential inclusion.

\end{proof}

\section{Simulation Study}

In this section, we evaluate the performance of the SP-ICL adaptive controller and study the effect of the sparsity regularization parameter $\lambda$. For the sake of this simulation, a nonlinear control-affine system of the form \eqref{eq:dynamics} is selected, with $Y$ defined as 
\begin{equation}
Y(x)=
\begin{bmatrix}
\phi(x)^\top & 0_{1\times10}\\
0_{1\times10} & \phi(x)^\top
\end{bmatrix}\in\mathbb{R}^{2\times20},
\end{equation}
where $\phi(x)=[1,x_1,x_2,x_1^2,x_1x_2,x_2^2,x_1^3,x_1^2x_2,x_1x_2^2,x_2^3]^{\!\top}$, sparse vector of true parameters is  $\theta = [0,\,-1,\,-1,\,0,\,0,\,0,\,0,\,0,\,0,\,0,\,
0,\,-\tfrac{1}{2},\,0,\,0,\,0,\,-\tfrac{1}{2},\,0,\,-\tfrac{1}{2}\\ ,\,0,\,0]^\top$. The control effectiveness matrix is selected as $g(x)=I_2$. 

The initial state is selected as $x(0) = [0.5\;0.5]^\top$, and the parameter estimate is initialized as $\hat{\theta}(t_{0}) = 0_{20 \times 1}$. The history stack is constructed online using the data selection procedure described in Algorithm~\ref{algo:dataselection}. The desired trajectory is defined as $x_d(t) = \begin{bmatrix} \sin(t)+0.12\sin(3t)-0.04\sin(5t) \\ 0.95\sin(2t)+0.08\sin(4t) \end{bmatrix}$. 

For the experiments, the selected proportional error gain is $K=10I_2$, the relative concurrent learning gain  $\gamma=0.1$, the adaptive update gain $\Gamma= I_{20}$, the minimum eigenvalue for the regressor  $\underline{y}=0.5$, and the normalization factor $\kappa= 0.01$. The integration window in Lemma~\ref{lem:ErrorTermformulation} is selected as $ T = 0.25$ seconds. The integration is performed for $100$ seconds using the \texttt{dde23} integrator in MATLAB.

Furthermore, to analyze the influence of the sparsity inducing term which is the primary contribution of the paper, We performed multiple experiments with $\lambda=[0, 10^{-5}, 10^{-4}, 10^{-3},5 \times 10^{-3}, 10^{-2}, 5\times 10^{-2}, 10^{-1}]$, and classified the parameter estimates into sparse and non-sparse classes, also producing a confusion matrix for the classification errors in the experiment.

\begin{figure}[t]
\centering
\begin{tikzpicture}
\begin{axis}[
    width=0.95\linewidth,
    height=0.70\linewidth,
    xlabel={$t\;(\mathrm{s})$},
    ylabel={$\|\tilde{\theta}(t)\|$},
    xmin=0, xmax=100,
    ymode=log,
    grid=major,
    grid style={line width=0.3pt, draw=gray!25},
    tick style={black},
    line cap=round,
    line join=round,
    label style={font=\small},
    tick label style={font=\small},
    legend style={
        font=\footnotesize,
        draw=none,
        fill=none,
        at={(0.5,-0.23)},
        anchor=north,
        legend columns=2,
        /tikz/every even column/.append style={column sep=0.8em},
        row sep=1pt,
    },
    axis line style={black},
    clip mode=individual,
]

\definecolor{c1}{RGB}{86,180,233}   
\definecolor{c2}{RGB}{230,159,0}    
\definecolor{c3}{RGB}{0,158,115}    
\definecolor{c4}{RGB}{204,121,167}  
\definecolor{c5}{RGB}{213,94,0}     
\definecolor{c6}{RGB}{240,228,66}   
\definecolor{c7}{RGB}{0,114,178}    
\definecolor{c8}{RGB}{0,150,150}  

\addplot[color=c1, line width=1.4pt]
table [x index=0, y index=1]
{results/lambda_0/data/parameter_estimation_error_norm.dat};
\addlegendentry{$\lambda=0$}

\addplot[color=c2, dashed, line width=1.4pt]
table [x index=0, y index=1]
{results/lambda_1em05/data/parameter_estimation_error_norm.dat};
\addlegendentry{$\lambda=10^{-5}$}

\addplot[color=c3, dotted, line width=1.4pt]
table [x index=0, y index=1]
{results/lambda_0p0001/data/parameter_estimation_error_norm.dat};
\addlegendentry{$\lambda=10^{-4}$}

\addplot[color=c4, dashdotted, line width=1.4pt]
table [x index=0, y index=1]
{results/lambda_0p001/data/parameter_estimation_error_norm.dat};
\addlegendentry{$\lambda=10^{-3}$}

\addplot[
    color=c5,
    line width=1.2pt,
    mark=o,
    mark size=1.6pt,
    mark repeat=12,
]
table [x index=0, y index=1]
{results/lambda_0p005/data/parameter_estimation_error_norm.dat};
\addlegendentry{$\lambda=5\times10^{-3}$}

\addplot[
    color=c6!80!black,
    line width=1.2pt,
    mark=square*,
    mark size=1.4pt,
    mark repeat=12,
]
table [x index=0, y index=1]
{results/lambda_0p01/data/parameter_estimation_error_norm.dat};
\addlegendentry{$\lambda=10^{-2}$}

\addplot[
    color=c7,
    line width=1.2pt,
    mark=triangle*,
    mark size=1.8pt,
    mark repeat=12,
]
table [x index=0, y index=1]
{results/lambda_0p05/data/parameter_estimation_error_norm.dat};
\addlegendentry{$\lambda=5\times10^{-2}$}

\addplot[
    color=c8,
    line width=1.2pt,
    mark=diamond*,
    mark size=1.6pt,
    mark repeat=12,
]
table [x index=0, y index=1]
{results/lambda_0p1/data/parameter_estimation_error_norm.dat};
\addlegendentry{$\lambda=10^{-1}$}

\end{axis}
\end{tikzpicture}
\caption{Parameter estimation error norm $\|\tilde{\theta}(t)\|$ (log scale) for different values of the regularization parameter $\lambda$.}
\label{fig:vdp_parameter_error_norm_all_lambda}
\end{figure}
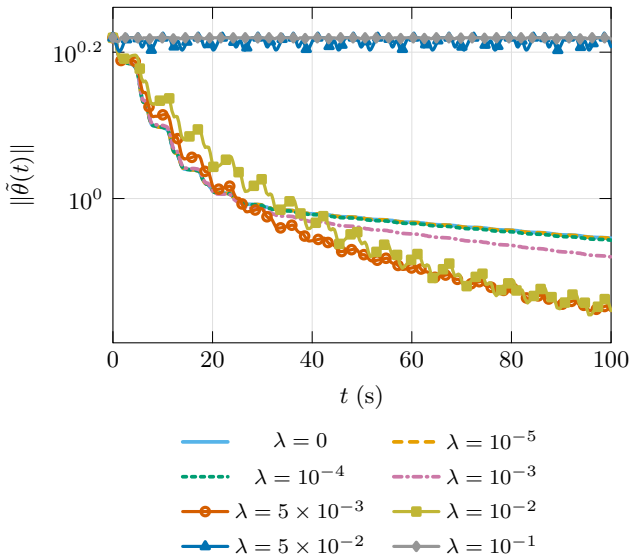

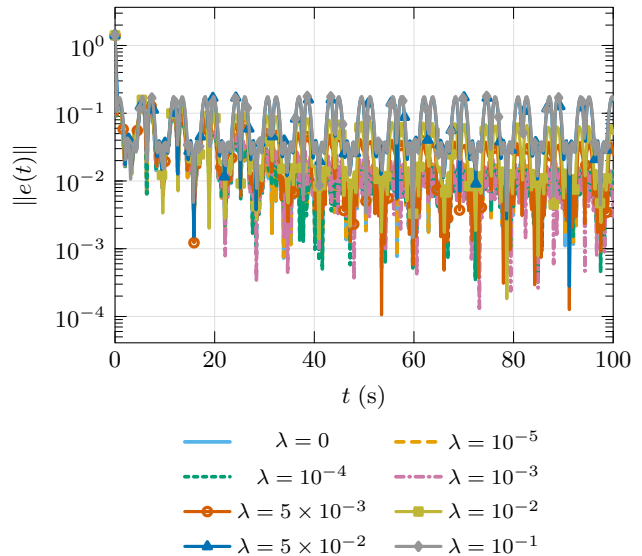
\begin{figure}[t]
\centering
\begin{tikzpicture}
\begin{axis}[
    width=0.95\linewidth,
    height=0.70\linewidth,
    xlabel={$t\;(\mathrm{s})$},
    ylabel={$\|e(t)\|$},
    xmin=0, xmax=100,
    ymode=log,
    grid=major,
    grid style={line width=0.3pt, draw=gray!25},
    tick style={black},
    line cap=round,
    line join=round,
    label style={font=\small},
    tick label style={font=\small},
    legend style={
        font=\footnotesize,
        draw=none,
        fill=none,
        at={(0.5,-0.23)},
        anchor=north,
        legend columns=2,
        /tikz/every even column/.append style={column sep=0.8em},
        row sep=1pt,
    },
    axis line style={black},
    clip mode=individual,
]

\definecolor{c1}{RGB}{86,180,233}
\definecolor{c2}{RGB}{230,159,0}
\definecolor{c3}{RGB}{0,158,115}
\definecolor{c4}{RGB}{204,121,167}
\definecolor{c5}{RGB}{213,94,0}
\definecolor{c6}{RGB}{240,228,66}
\definecolor{c7}{RGB}{0,114,178}
\definecolor{c8}{RGB}{0,150,150}

\addplot[
    color=c1,
    line width=1.4pt
]
table [x index=0, y index=1]
{results/lambda_0/data/tracking_error_norm.dat};
\addlegendentry{$\lambda=0$}

\addplot[
    color=c2,
    dashed,
    line width=1.4pt
]
table [x index=0, y index=1]
{results/lambda_1em05/data/tracking_error_norm.dat};
\addlegendentry{$\lambda=10^{-5}$}

\addplot[
    color=c3,
    dotted,
    line width=1.4pt
]
table [x index=0, y index=1]
{results/lambda_0p0001/data/tracking_error_norm.dat};
\addlegendentry{$\lambda=10^{-4}$}

\addplot[
    color=c4,
    dashdotted,
    line width=1.4pt
]
table [x index=0, y index=1]
{results/lambda_0p001/data/tracking_error_norm.dat};
\addlegendentry{$\lambda=10^{-3}$}

\addplot[
    color=c5,
    line width=1.2pt,
    mark=o,
    mark size=1.6pt,
    mark repeat=12
]
table [x index=0, y index=1]
{results/lambda_0p005/data/tracking_error_norm.dat};
\addlegendentry{$\lambda=5\times10^{-3}$}

\addplot[
    color=c6!80!black,
    line width=1.2pt,
    mark=square*,
    mark size=1.4pt,
    mark repeat=12
]
table [x index=0, y index=1]
{results/lambda_0p01/data/tracking_error_norm.dat};
\addlegendentry{$\lambda=10^{-2}$}

\addplot[
    color=c7,
    line width=1.2pt,
    mark=triangle*,
    mark size=1.8pt,
    mark repeat=12
]
table [x index=0, y index=1]
{results/lambda_0p05/data/tracking_error_norm.dat};
\addlegendentry{$\lambda=5\times10^{-2}$}

\addplot[
    color=c8,
    line width=1.2pt,
    mark=diamond*,
    mark size=1.6pt,
    mark repeat=12
]
table [x index=0, y index=1]
{results/lambda_0p1/data/tracking_error_norm.dat};
\addlegendentry{$\lambda=10^{-1}$}
\end{axis}
\end{tikzpicture}
\caption{Tracking error norm $\|e(t)\|$ (log scale) for different values of the regularization parameter $\lambda$.}
\label{fig:vdp_tracking_error_norm_all_lambda}
\end{figure}

\section{Discussion}
Figures~\ref{fig:vdp_parameter_error_norm_all_lambda} and \ref{fig:vdp_tracking_error_norm_all_lambda} show that the parameter estimation error and tracking error converge to bounded neighborhood under the developed controller in \eqref{eq:control} and parameter update law in \eqref{eq:theta_dot}. These results are consistent with the guarantees established in Theorem~\ref{thm:stability}.

An important observation in Figures~\ref{fig:vdp_parameter_error_norm_all_lambda} and \ref{fig:vdp_tracking_error_norm_all_lambda} is the presence of chattering induced by the discontinuous signum term in the update law. This effect is further exacerbated by finite sampling and becomes more pronounced as the sparsity parameter $\lambda$ increases. For smaller values of $\lambda \in \{10^{-5}, 10^{-4}, 10^{-3}\}$, the chattering  amplitude remains relatively low. In contrast, higher sparsity levels, corresponding to $\lambda \in \{5\times10^{-3}, 10^{-2},5\times10^{-2}, 10^{-1}\}$, result in increased chattering, with $\lambda = 10^{-1}$ exhibiting the largest amplitude.

Additionally, as shown in Table~\ref{tab:main_results}, the parameter estimation error decreases with increasing $\lambda$, indicating greater sparsity in the learned model. This behavior can be attributed to the reduction of spurious correlations among candidate basis functions, leading to a more accurate identification of the true underlying dynamics. For $\lambda \geq 10^{-2}$, the error begins to increase, indicating that excessive sparsity penalization can remove relevant basis functions and reduce model fidelity. However, the tracking error increases with stronger sparsity regularization. This degradation in tracking performance may be attributed to higher-amplitude chattering induced by the non-smooth parameter updates.

    Table~\ref{tab:confusion_results} and Figure~\ref{fig:confusionmatrix} show that sparsity increases, i.e., the number of nonzero terms decreases, as $\lambda$ increases, as expected. This effect is more pronounced in the interval $\lambda \in [10^{-3}, 10^{-2}]$, where a significant increase in sparsity is observed compared to smaller values of $\lambda$. This trend is also reflected in Fig.~\ref{fig:vdp_parameter_error_norm_all_lambda}, where the parameter estimation error is lower for $\lambda$ in this range. Furthermore, Table~\ref{tab:confusion_results} indicates that very large values of $\lambda$ result in fewer positive terms than required, leading to bias in the identified dynamics. This highlights a tradeoff between sparsity and estimation accuracy. We treat sparsity as a binary classification problem, labeling each basis function as non-sparse (positive) or sparse (negative), and evaluate performance using the F-1 score \cite[Chapter~8]{SCC.Manning.Raghavan2009}. From Table~\ref{tab:confusion_results}, the F-1 score is 0.59 for small $\lambda$, where correlated functions are also positive. As $\lambda$ increases, precision improves while recall remains high, giving a maximum F-1 score of 0.89 at $\lambda = 10^{-2}$. For larger $\lambda$, true positives are removed, reducing recall and driving the F-1 score to zero. For this example, the best performance is obtained at $\lambda = 5 \times 10^{-2}$ and the best system was identified at $\lambda = 10^{-2}$.

\begin{table}[t]
\centering
\caption{Effect of sparsity regularization on tracking performance and model recovery. Lower is better for error metrics. Final Time {$t_f=100 \,sec$}}
\label{tab:main_results}
\setlength{\tabcolsep}{5pt}
\renewcommand{\arraystretch}{1.15}

\begin{tabular}{c|c|cc}
\toprule
& \textbf{Non-Sparse terms} & \multicolumn{2}{c}{\textbf{Error}} \\
$\lambda$ & Nonzeros $\downarrow$ & $RMSE(\|e\|)$ $\downarrow$ & $\|\tilde{\theta}(t_f)\|$ $\downarrow$ \\
\midrule

0 
& 12 & \textbf{0.04452} & 0.88263 \\

$10^{-5}$ 
& 12 & 0.04453 & 0.88209 \\

$10^{-4}$ 
& 12 & 0.04456 & 0.87730 \\

$10^{-3}$ 
& 12 & 0.04507 & 0.83245 \\

$5\times10^{-3}$ 
& 8  & 0.05013 & \textbf{0.70261} \\

$10^{-2}$ 
& \textbf{4}  & 0.05973 & 0.70286 \\

$5\times10^{-2}$ 
& 1  & 0.09927 & 1.61225 \\

$10^{-1}$ 
& \textbf{0}  & 0.10226 & 1.65263 \\

\bottomrule
\end{tabular}
\end{table}

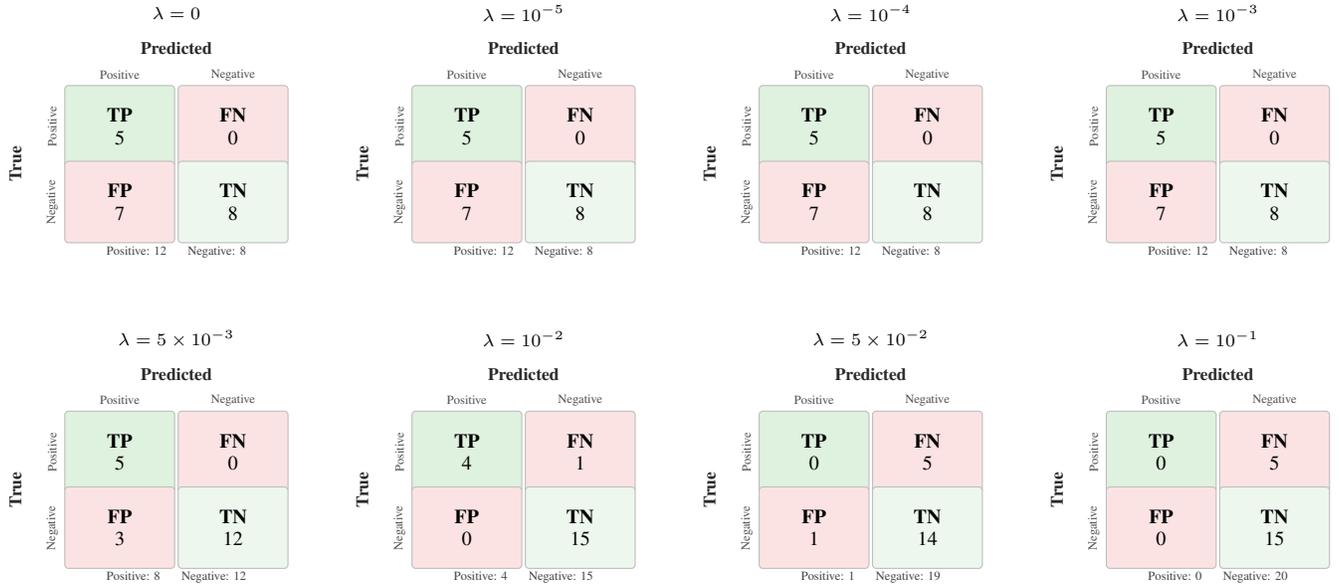
\begin{figure*}[t]
    \centering
    \resizebox{\linewidth}{!}{%
\begin{tikzpicture}[
    font=\footnotesize,
    cell/.style={
        draw=black!25,
        line width=0.45pt,
        minimum width=1.55cm,
        minimum height=1.15cm,
        align=center,
        inner sep=2pt,
        rounded corners=1.5pt
    },
    head/.style={font=\scriptsize\bfseries, text=cmtext},
    sub/.style={font=\tiny, text=black!70},
    lam/.style={font=\scriptsize\bfseries},
    total/.style={font=\tiny, text=black!75},
]

\newcommand{\cmpanel}[9]{%
\begin{scope}[shift={(#1,#2)}]
    \node[lam] at (1.55,2.95) {$\lambda=#3$};

    \node[head] at (1.55,2.45) {Predicted};
    \node[sub]  at (0.75,2.08) {Positive};
    \node[sub]  at (2.35,2.08) {Negative};

    \node[head, rotate=90] at (-0.72,0.82) {True};
    \node[sub, rotate=90] at (-0.20,1.35) {Positive};
    \node[sub, rotate=90] at (-0.20,0.28) {Negative};

    \node[cell, fill=cmgood!18] at (0.75,1.35) {\textbf{TP}\\#4};
    \node[cell, fill=cmbad!16]  at (2.35,1.35) {\textbf{FN}\\#6};
    \node[cell, fill=cmbad!16]  at (0.75,0.28) {\textbf{FP}\\#5};
    \node[cell, fill=cmgood!10] at (2.35,0.28) {\textbf{TN}\\#7};

    \node[total] at (1.55,-0.42) {Positive: #8 \hspace{1em} Negative: #9};
\end{scope}
}

\cmpanel{0.0}{0.0}{0}{5}{7}{0}{8}{12}{8}
\cmpanel{4.9}{0.0}{10^{-5}}{5}{7}{0}{8}{12}{8}
\cmpanel{9.8}{0.0}{10^{-4}}{5}{7}{0}{8}{12}{8}
\cmpanel{14.7}{0.0}{10^{-3}}{5}{7}{0}{8}{12}{8}

\cmpanel{0.0}{-4.0}{5\times10^{-3}}{5}{3}{0}{12}{8}{12}
\cmpanel{4.9}{-4.0}{10^{-2}}{4}{0}{1}{15}{4}{15}
\cmpanel{9.8}{-4.0}{5\times10^{-2}}{0}{1}{5}{14}{1}{19}
\cmpanel{14.7}{-4.0}{10^{-1}}{0}{0}{5}{15}{0}{20}

\end{tikzpicture}
}
    \caption{Confusion matrix statistics for sparse term recovery. Classification performance is reported for nonzero (positive) and zero (negative) terms using true positives (TP), false positives (FP), false negatives (FN), and true negatives (TN).}
    \label{fig:confusionmatrix}
\end{figure*}

\begin{table}[t]
\centering
\caption{Non-sparse term recovery performance. Confusion matrix counts and derived Precision(Prec.), Recall(Rec.) and F-1 metrics for nonzero term detection.}
\label{tab:confusion_results}
\setlength{\tabcolsep}{5pt}
\renewcommand{\arraystretch}{1.15}

\begin{tabular}{c|cc|ccc}
\toprule
& \multicolumn{2}{c|}{\textbf{Counts}} 
& \multicolumn{3}{c}{\textbf{Metrics}} \\
$\lambda$ 
& Positive & Negative 
& Prec & Rec & F-1 \\
\midrule

0 
& 12 & 8 
& 0.42 & 1.00 & 0.59 \\

$10^{-5}$ 
& 12 & 8 
& 0.42 & 1.00 & 0.59 \\

$10^{-4}$ 
& 12 & 8 
& 0.42 & 1.00 & 0.59 \\

$10^{-3}$ 
& 12 & 8 
& 0.42 & 1.00 & 0.59 \\

$5\times10^{-3}$ 
& 8 & 12 
& 0.62 & 1.00 & 0.77 \\

$10^{-2}$ 
& 4 & 16 
& \textbf{1.00} & 0.80 & \textbf{0.89} \\

$5\times10^{-2}$ 
& 1 & 19 
& 0.00 & 0.00 & 0.00 \\

$10^{-1}$ 
& 0 & 20 
& -- & 0.00 & 0.00 \\

\bottomrule
\end{tabular}
\end{table}

\section{Conclusion}
This paper presents an online adaptive control framework with sparse identification that integrates $\ell_1$-regularization with ICL and trajectory tracking for nonlinear control-affine systems. The inclusion of $\ell_1$-regularization penalizes nonzero terms, reducing the co-occurrence of correlated basis functions and thereby promoting sparsity. The SP-ICL update law identifies relevant basis functions from an overcomplete dictionary of candidate basis functions. The stability of the resulting framework is established via Lyapunov-based analysis for non-smooth systems, where we prove ultimate boundedness of the trajectories of the closed-loop system while promoting sparsity.

Simulation results show that sparsity improves parameter convergence by penalizing basis functions that are not present in the dynamical system, while still preserving stability and tracking accuracy. To mitigate the observed chattering, future research could explore replacing the discontinuous signum term in the SP-ICL update law with continuous approximations, such as saturation or smooth hyperbolic tangent functions. Furthermore, future work would explore adaptive $\ell_{1}$ regularization penalties to dynamically balance sparsity promotion and estimation accuracy, which could improve convergence for a broader class of nonlinear systems.

\small
\bibliographystyle{IEEETrans.bst}
\bibliography{scc, sccmaster, scctemp, mybib, cdc2026}

\end{document}